\documentclass[11pt,reqno]{amsart}
\usepackage{amsbsy,amsfonts,amsmath,amssymb,amscd,amsthm}
\usepackage{mathrsfs}
\usepackage{subfigure,enumitem,color,graphicx}
\usepackage[a4paper,text={6.1in,8in},centering,includefoot,foot=1in]{geometry}
\usepackage{pxfonts}
\usepackage[colorlinks=true,linkcolor=blue,citecolor=green]{hyperref}
\usepackage{srcltx}

\small\normalsize
\newtheorem{theorem}{Theorem}[section]

\newtheorem{lemma}[theorem]{Lemma}
\newtheorem{proposition}{Proposition}

\newtheorem{claim}{Claim}
\theoremstyle{definition}

\newtheorem{remark}{Remark}

\newtheorem{mainthm}{Theorem}
\newtheorem{maincor}[mainthm]{Corollary}

\DeclareMathOperator{\IFS}{IFS}
\newcommand{\arXiv}[1]{\rm{arXiv} #1}

\begin{document}

\title[Fiberwise orbits in minimal IFSs]
      {Density of fiberwise orbits in minimal iterated function systems on the circle}

\author[Barrientos]{Pablo G. Barrientos}
\address{\footnotesize \centerline{Instituto de Matem\'atica e Estat\'istica, UFF}
   \centerline{Rua M\'ario Santos Braga s/n - Campus Valonguinhos, Niter\'oi,  Brazil}}
\email{barrientos@id.uff.br}

\author[Fakari]{Abbas Fakhari}
\address{\footnotesize Department of Mathematics, Shahid Beheshti University \\
G.C.Tehran 19839, Tehran, Iran \\ a\_fakhari@sbu.ac.ir}

\author{A. Sarizadeh}
\address{Department of Mathematics, Ilam University, Ilam, Iran}
\email{a.sarizadeh@mail.ilam.ac.ir}


\begin{abstract}
We study the minimality of almost every orbital branch of minimal
iterated function systems (IFSs). We prove that this kind of
minimality holds for forward and backward minimal IFSs generated
by orientation-preserving homeomorphisms of the circle. We provide
new examples of iterated functions systems where this behavior
persists under perturbation of the generators.
\end{abstract}

\maketitle

\section{Introduction}
\label{s:1} Group and semigroup actions on the circle is the main
subject of recent studies and still attract lots of attention.
Much of these efforts have been focused so far to explore the rich
dynamics of finitely generated actions, i.e. dynamics generated by
finitely many maps. Here, with regard to this issue, we consider
dynamical systems generated by several maps on a compact metric
space X, called iterated function systems (IFSs). More precisely,
given maps $f_1, \dots,f_k$ of $X$, we study the action of the
semigroup $\IFS(f_1,\dots,f_k)$ generated by these maps. The orbit
of a point $x$ for $\IFS(f_1,\ldots,f_k)$ is a set of points
$y=h(x)$, for some $h\in\IFS(f_1,\ldots,f_k)$. The IFS is
\emph{minimal} if every orbit is dense in $X$. A sequence of
iterates $x_{n+1}=f_{\omega_n}(x_n)$ with $\omega_n \in
\{1,\dots,k\}$ chosen randomly and independently is called
\emph{branch orbit} starting at $x_0=x$. The sequence of
compositions
$$
   f_{\omega}^n=f_{\omega_n}\circ\dots \circ f_{\omega_1}, \quad \text{for every $n\in
   \mathbb{N}$},
$$
is called \emph{orbital branch} corresponding to
$\omega=\omega_1\omega_2\dots\in\Sigma_k^+=\{1,\dots,k\}^\mathbb{N}$
and
$$
\mathcal{O}^+_\omega(x)=\{f_\omega^n(x): n\in \mathbb{N}\}, \quad
x\in X
$$
is the \emph{fiberwise orbit}. Following \cite{BV11} we consider
any probability $\mathbb{P}$ on $\Sigma_k^+$ with the following
property:  there exists $0<p\leq 1/k$ so that $\omega_n$ is
selected randomly from $\{1,\dots,k\}$ in such a way that the
probability of $\omega_n=i$ is greater or equal to $p$,
for every $i\in \{1,\dots,k\}$ and every $n\in \mathbb{N}$. More
formally, in terms of conditional probability,
$$
\mathbb{P}( \omega_n=i \ | \ \omega_{n-1},\ldots,\omega_1 ) \geq
p.
$$
Observe that the standard Bernoulli measures on $\Sigma_k^+$ are
typical examples of these kind of probabilities.

In the recent works~\cite{BV11,BV12,BR13}, the limit set of IFSs
have been studied providing some conditions to guarantee
minimality. In the present paper, focusing essentially on the same
subject, we study the density of fiberwise orbits of minimal IFSs,
mainly on the circle. This goal is motivated in part by the main
result in~\cite{BV11} saying that for a minimal IFS, almost every
branch orbit starting from an arbitrary point is dense in $X$. To
be more precise, for each point $x$ there exists a set $\Omega(x)
\subset \Sigma_k^+$ with $\mathbb{P}(\Omega(x))=1$ such that
$\mathcal{O}^+_\omega(x)$ is dense in $X$, for every $\omega\in
\Omega(x)$. A logical question would be whether the  set
$\Omega(x)$ could be independent of the choice of $x$. In other
words, if almost every orbital branch of a minimal IFS acts
minimally, i.e.,
\begin{equation}
\label{eq:1} X=\overline{\mathcal{O}^+_\omega(x)}, \quad \text{for
every}\ x\in X\ \text{and}\ \text{for almost every $\omega\in
\Sigma_k^+$}.
\end{equation}
We obtain an answer to this question when the ambient space is the
circle.

There are several obstacles and solutions to overcome the
dependence of $\Omega(x)$ to the initial point $x$ in general
case. For instance, this is done in~\cite{BL13} under a condition
named {\it strongly-fibred}. More precisely, the authors have
shown that when a minimal IFS is strongly-fibred then~\eqref{eq:1}
holds for 
every $\omega \in \Sigma_k^+$ with dense orbit under the Bernoulli
shift map. Recall that, a minimal IFS is ``strongly-fibred'' if
for every open set $U$ in $X$ there exists $\omega\in\Sigma_k^+$
such that $X_\omega \subset U$, where
$$
   X_\omega= \bigcap_{n=1}^\infty \hat{f}^n_{\omega}(X),
   \qquad \hat{f}_\omega^n=f_{\omega_1}
   \circ\dots\circ f_{\omega_{n-1}}\circ f_{\omega_n}.
$$
Historical examples of strongly-fibred minimal IFSs are those
generated by contraction maps. In this case, there is a wealth of
classical results, essentially going back to the seminal work of
Hutchinson~\cite{Hut81}. Meanwhile, a \emph{weakly hyperbolic} IFS
is one for which
$$
\lim_{n\to \infty} \mathrm{diam}(\hat{f}^n_\omega(X))=0, \quad
\text{for every $\omega \in \Sigma_k^+$.}
$$
It turns out that the minimal weakly-hyperbolic IFSs are
strongly-fibred. Weakly hyperbolic IFS characterized
in~\cite{AJS12} as those for which 
\begin{equation}
\label{eq:contracting} \lim_{n\to \infty}
\sup_{\omega\in\Sigma_k^+} d(f^n_{\omega}(x),f^n_\omega(y))=0
\quad \text{for every $x,y\in X$.}
\end{equation}

A point $q\in X$ is a \emph{(repelling/attracting) periodic point}
for $\IFS(f_1,\dots,f_k)$ if there exists
$h\in\IFS(f_1,\dots,f_k)$ such that $q$ is a
(repelling/attracting) fixed point of $h$. Clearly, weak
hyperbolicity prevents the existence of repelling periodic points.
A weaker property than~\eqref{eq:contracting} which allows the
existence of repelling periodic points for the IFS is
``contraction of almost every orbital branch'' which was
introduced as \emph{synchronization} in~\cite{hom13}. This means
that there exists a set $\Omega \subset \Sigma_k^+$ with
$\mathbb{P}(\Omega)=1$ such that for every $\omega\in \Omega$,
there is a dense set $W(\omega) \subset X$ with
\begin{equation*}
\lim_{n\to\infty}d(f^n_\omega(x),f^n_\omega(y))=0, \quad \text{for
every $x,y\in W(\omega)$}.
\end{equation*}
Examples of minimal IFSs on the circle satisfying this kind of
contracting property can be found in~\cite[Theorem~1]{KN04} and
\cite[Theorem~1]{Hom11}. These examples assume, among other
things, that the IFS is \emph{forward and backward minimal}. That
is,  the semigroup generated by $f_1,\dots,f_k$ acts minimally on
the circle and the same holds for the semigroup generated by their
inverses. In our first result,
we show the minimality of almost every orbital branch and the
density of periodic points for these kind of IFSs.
\begin{mainthm}
\label{thmB} \emph{Let $f_1,\dots,f_k$ be orientation-preserving
homeomorphisms of the circle. Assume that the IFS generated by
these maps is forward and backward minimal.
Then there exists $\Omega \subset \Sigma_k^+$ with
$\mathbb{P}(\Omega)=1$ such that
$$
    \mathbb{S}^1= \overline{\mathcal{O}^+_\omega(x)}, \quad  \text{for every $x \in S^1$ and $\omega \in \Omega$.}
$$
In particular,  $\Omega$ contains all the sequences with dense
orbit under the Bernoulli shift map and if there exists a
homeomorphism in $\IFS(f_1,\dots,f_k)$ which is not conjugate to a
rotation, then the periodic periodic points of
$\IFS(f_1,\dots,f_k)$ are dense in $\mathbb{S}^1$.}
\end{mainthm}
Observe that the minimal IFSs mentioned in this theorem above are
not strongly-fibred, as $X_\omega=\mathbb{S}^1$, for every
$\omega\in\Sigma_k^+$. Thus, Theorem \ref{thmB} provides a new
condition under which a minimal IFS satisfies~\eqref{eq:1}.

While,
there are a variety of recent concrete examples  providing IFSs
satisfying the assumption of Theorem~\ref{thmB}
(\cite{BR13,KN04}), the theorem is in a way restrictive: only the
circle is discussed. The main place in the proof where this
assumption is needed is Antonov's theorem stated on the circle. We
try to overcome this limitation by providing examples of minimal
IFSs satisfying~\eqref{eq:1} directly. Namely, we prove
\eqref{eq:1} in general case for the IFSs containing a minimal
homeomorphism (see Proposition~\ref{pro:minimal-map}).

Our next goal is to build IFSs on the circle satisfying the
assumptions of Theorem~\ref{thmB} in a robust way. Here, the
robustness is understood  as the persistence of the minimality
under $C^1$-perturbations of the generators of the initial IFS.
There are various ways to construct such IFSs. For instance, it is
shown in~\cite{BR13} that every IFS generated by a pair of
diffeomorphisms $C^2$-close enough to rotations with no periodic
points in common and without periodic $ss$-intervals (compact
intervals whose endpoints are consecutive attracting periodic
points of different generators) is forward and backward minimal in
a robust way. However, the first examples of $C^1$-robustly
forward and backward minimal IFSs on the circle going back to
\cite{GI99,GI00}. These examples require that the IFS contains an
\emph{irrational rotation} (i.e., a  rigid rotation of the circle
with irrational rotation number) and a $C^1$-diffeomorphism $g$
with an attracting hyperbolic fixed point $a$ with derivative
$Dg(a)$ lying in the interval $(1/2,\,1)$. In our next result, we
generalize these kind of examples removing the extra assumptions
on he fixed point of~$g$.
\begin{mainthm}
\label{thmC} \emph{Let $g_1, g_2 \in
\mathrm{Diff}^1(\mathbb{S}^1)$ be, respectively, an irrational
rotation and an orientation-preserving diffeomorphism
which is not conjugate to a rotation.
Then,
there exists a $C^1$-neighborhood $\mathcal{U}$ of $(g_1,g_2)$ 
such that the IFS generated by any pair $(f_1,f_2)\in \mathcal{U}$
is forward and backward minimal.
Consequently, the following hold
\begin{itemize}
\item \emph{minimality of almost every orbital branch}:
$ \mathbb{S}^1= \overline{\mathcal{O}^+_\omega(x)}, \text{for
every}$
 $x\in \mathbb{S}^1$
and $\omega\in\Sigma_2^+$ with dense orbit under the Bernoulli shift map. \\[-0.3cm]
\item \emph{density of periodic points}:
the periodic points of $\IFS(f_1,f_2)$ are dense in
$\mathbb{S}^1$. Moreover, if the IFS has a hyperbolic
attracting/repelling
 periodic point then it has dense set of hyperbolic
attracting/repelling periodic points.
\end{itemize}}
\end{mainthm}
An immediate consequence is that every IFS with an irrational
rotation can be approximated by $C^1$-robustly forward and
backward minimal one. If the IFS contains a diffeomorphism which
is conjugate to an irrational rotation, one can conjugate the IFS
to one containing this irrational rotation. Unluckily, we cannot
use the previous theorem since the conjugacy map is not
differentiable in general.
However, the following result shows that, even in this case, the
IFS can be approximated by a $C^1$-robustly forward and backward
minimal one. The main tool applied here is the smooth conjugacy
for the circle diffeomorphisms provided in~\cite{BG12}.
\begin{maincor}
\label{corA} \emph{Let $g_1,\dots,g_k$ be $C^1$-diffeomorphisms of
the circle with $k\geq 2$ and suppose that there is a map in
$\IFS(g_1,\dots,g_k)$ with irrational rotation number. Then for
every $C^1$-neighborhood $\mathcal{U}$ of $(g_1,\dots,g_k)$ there
is $(f_1,\dots,f_k)\in \mathcal{U}$ such that the IFS generated by
these maps is $C^1$-robustly forward and backward minimal.}
\end{maincor}
\begin{proof}
Let $f$ be the $C^1$-diffeomorphisms in $\IFS(f_1,\dots,f_k)$ with
irrational rotation number $\alpha$. According to~\cite{BG12},
there exists a sequence of $C^1$-diffeomorphisms $h_n$ such that
$h_n \circ f \circ h_n^{-1}$ tends to the rotation $R_\alpha$ in
the $C^1$-topology as $n \to \infty$. Thus,  $F_n =h_n\circ
\IFS(f_1,\dots,f_k)\circ h^{-1}_n$ can be taken arbitrarily
$C^1$-close to an IFS containing $R_\alpha$. Then, if necessary,
by means of a small perturbation of the generators $f_1,\dots,f_k$
we get a semigroup $G_n$ arbitrarily close to $F_n$ satisfying
Theorem~\ref{thmC}. Consequently, the IFS given by $h_n^{-1}\circ
G_n \circ h_n$ is $C^1$-robustly forward and backward minimal and
arbitrarily close to $\IFS(f_1,\dots,f_k)$. This concludes the
proof of the result.
\end{proof}
Let us end this introduction by asking if every minimal IFS of
$C^1$-diffeo\-mor\-phisms, even on the circle, can be approximated
by $C^1$-robustly minimal one.

The proof of Theorem~\ref{thmB} is handled in Secion 2. Section 3
is devoted to the proof of Theorem~\ref{thmC} and some auxiliary
lemmas.


\section{Minimality of orbital branches: Proof of Theorem~\ref{thmB}}
\label{s:3} We first study the minimality of almost every orbital
branch in the special case that the IFS contains a minimal
homeomorphism. To do this, we need a bit of notation. Given a
finite word $\sigma=\sigma_1\dots\sigma_n$ in the alphabet
$\{1,\dots,k\}$ we denote by $|\sigma|$ the length of~$\sigma$.

\begin{proposition}
\label{pro:minimal-map} Let $f_1,\dots,f_k$  be homeomorphisms of
a compact metric space $X$. Assume that $\IFS(f_1,\dots,f_k)$
contains a minimal homeomorphism. Then there exists $\Omega
\subset \Sigma_k^+$ with $\mathbb{P}(\Omega)=1$ such that
$$
X= \overline{\mathcal{O}^+_\omega(x)}, \quad  \text{for every
$x\in X$ and $\omega \in \Omega$.}
$$
Moreover, $\Omega$ contains all the sequences with dense orbit
under the shift map.
\end{proposition}
\begin{proof}
By the compactness of $X$, there is a countable open base
$\mathcal{B}$. Let $B$ be an element of $\mathcal{B}$ and fix
$x\in X$.
\begin{claim}
\label{clm1} There is $\Omega(B)\subset \Sigma_k^+$, with
$\mathbb{P}(\Omega(B))=1$, such that for every $\omega\in
\Omega(B)$,
\begin{equation}
\label{eq:provar-esto} f^n_\omega(x)\in B, \quad \text{for some
$n\geq 1$}.
\end{equation}
Moreover, $\Omega(B)$ contains all the sequences with dense orbit
under the shift map.
\end{claim}
By doing this, the proof of the proposition can be derived
from the countability of $\mathcal{B}$. Indeed, it follows that
$$
   \Omega = \bigcap_{B \in \mathcal{B}} \Omega(B)
$$
has full $\mathbb{P}$-measure, contains all the sequences with
dense orbit under the shift map and it holds that for every
$\omega \in \Omega$ and $B \in \mathcal{B}$,
$$
\mathcal{O}_\omega^+(x) \cap B \not = \emptyset \quad \text{for
every $x\in X$.}
$$
This completes the proof of the proposition.
\end{proof}
\begin{proof}[Proof of Claim \ref{clm1}]
We first provide $\ell\in \mathbb{N}$ and a word $\sigma$ with
$|\sigma|=\ell$ in such a way that for every
$\omega\in\mathcal{C}_{\sigma}$ and $x\in X$,
\begin{equation}
\label{eq:1step} f^t_{\omega} (x)\in B \quad \text{for some
$0<t\leq \ell$,}
\end{equation}
where $\mathcal{C}_\sigma$ denotes the cylinder in $\Sigma_k^+$
around the finite word $\sigma=\sigma_1\dots\sigma_\ell$. Let
$h=f_{\alpha_s}\circ\dots\circ f_{\alpha_1}$ be the minimal
homeomorphism in $\IFS(f_1,\ldots,f_k)$. Observe that the
minimality $h$ is equivalent to that of $h^{-1}$, so
$$
X = \bigcup_{i=1}^r  h^{-i}(B) \quad \text{for some $r\in
\mathbb{N}$.}
$$
Hence, for each $x\in X$, there is $ i\in \{1,\dots,r\}$ with $
h^i(x)\in B. $ Put $\ell=rs$. Thus, defining
$\sigma=\alpha\dots\alpha$ where $\alpha=\alpha_1\dots\alpha_s$
appears $r$-times
one can conclude~\eqref{eq:1step}.

Continuing the proof of Claim \ref{clm1}, let $\Gamma(B)$ be the
set of elements $\omega\in\Sigma_k^+$ for which there exists $x\in
X$ such that $f^j_\omega(x)\not\in B$, for all $j\in\mathbb{N}$.
Observe that
\begin{equation}
\label{eq:2step} \Gamma(B)\subset \bigcap_{n=1}^\infty
\Gamma(B,n),
\end{equation}
where $ \Gamma(B,n)=\{\omega: \exists\, x\in X \ \text{such that}
\  f ^j_\omega(x)\not\in B  \ \text{for every} \ j\leq n+\ell\}$.
Equations~\eqref{eq:1step} and~\eqref{eq:2step} imply that
\begin{align*}
\mathbb{P}(\Gamma(B)) &\leq
\mathbb{P}(\Gamma(B,n)) \\
&\leq (1-p^\ell) \cdot \mathbb{P}(\Gamma(B,n-\ell)) \\
&\leq (1-p^\ell)^{1+[n/\ell]} \to 0 \quad \text{when $n\to
\infty$}.
\end{align*}
Thus, taking $\Omega(B)=\Sigma_k^+\setminus \Gamma(B)$, one gets
$\mathbb{P}(\Omega(B))=1$ which implies ~\eqref{eq:provar-esto}.
Observe that by~\eqref{eq:1step} any sequence with dense orbit
under the shift map belongs to $\Omega(B)$.  This concludes the
proof of Claim \ref{clm1}.
\end{proof}
The next theorem establishing the synchronization phenomenon is
the key ingredient to prove Theorem~\ref{thmB}. We denote by $\nu$
any Bernoulli measure on $\Sigma_k^+$.
\begin{theorem}[Antonov~\cite{An84}]
\label{thm:Antonov}Let $f_1,\dots,f_k$ be orientation-preserving
homeomorphisms of the circle such that the IFS generated by them
is forward and backward minimal. Then exactly one of the following
statements holds:
\begin{enumerate}
\item \label{item11} there exists a common invariant measure of all the $f_i$'s, and all these maps are simultaneously conjugate to rotations;
\item\label{item22} for any two points $x,y\in \mathbb{S}^1$, there exists
$\Omega=\Omega(x,y)\subset \Sigma_k^+$ with $\nu(\Omega)=1$ such
that
    $$
      \lim_{n\to\infty} d(f^n_\omega(x),f^n_\omega(y))=0, \quad \text{for every $\omega \in \Omega$};
    $$
\item \label{item33} there exists an integer $\ell>1$ and an order
$\ell$ orientation-preserving homeomorphism $\varphi$, such that
it commutes with all the $f_i$'s, and after passing to the
quotient circle
$$\mathbb{S}^1 /(\varphi^i(x)\sim \varphi^j(x))$$
the conclusion of~\eqref{item22} are satisfied, for the new maps
$g_i$.
\end{enumerate}
\end{theorem}
The next lemma which is essentially consequence of Antonov's
theorem plays a key role in proving Theorem \ref{thmB}. Our
arguments rely upon
the same basic strategy as those of~\cite
{GGKV13}.
\begin{lemma}
\label{lem:Antonov} Under the assumptions of
Theorem~\ref{thm:Antonov}, if there is a homeomorphism in
$\IFS(f_1,\dots,f_k)$ which is not conjugate to a rotation then
there exists an integer $\ell\geq 1$ such that for $\nu$-almost
every $\omega\in \Sigma_k^+$ there are points $r_i=r_i(\omega)\in
\mathbb{S}^1$, for $i=1,\dots,\ell$, such that
$$
  \lim_{n\to\infty} d(f^n_\omega(x),f^n_\omega(y))=0, \quad \text{for every $x,y\in U$.}
$$
where $U$ is any connected component of
$\mathbb{S}^1\setminus\{r_1,\dots,r_\ell\}$.
\end{lemma}
\begin{proof}
Assume that we are in the case~\eqref{item22} of
Theorem~\ref{thm:Antonov}. For any arc $I\subset \mathbb{S}^1$
there is a subset $\Omega(I)\subset \Sigma^+_k$ with
$\nu(\Omega(I))=1$ such that either the length of arc
$f^n_\omega(I)$ or its complement
$f_\omega^n(\mathbb{S}^1\setminus I)$ tends to 0 for every
$\omega\in \Omega(I)$. Consider a sequences of finer partitions
$\mathcal{P}_m$ of the circle into closed arcs with rational
endpoints whose length goes to zero. Since we have a numerable
number of the arcs arcs
$$
\Omega = \bigcap_{m\in\mathbb{N}} \bigcap_{I\in \mathcal{P}_m}
\Omega(I)\subset   \Sigma_k^+
$$
satisfies $\nu(\Omega)=1$. For each $\omega\in \Omega$, as the
total length of the circle is preserved, there is exactly one arc
$I_m=I_m(\omega)$ of the partition $\mathcal{P}_m$ such that the
length of $f^n_\omega(I_m)$ tends to 1 and the length of the
images of its complement tends to 0. These closed arcs $I_m$,
$m\in \mathbb{N}$ form a nested sequence whose length goes 0, and
so there exists a unique intersection point $r(\omega)\in
\mathbb{S}^1$. Therefore, for every two points $x,y \in
\mathbb{S}^1\setminus \{r(\omega)\}$, here is a natural number $m$
such that $x,y \in \mathbb{S}^1\setminus I_m$, and hence
$$
\lim_{n\to\infty}d(f^n_\omega(x),f^n_\omega(y))=0.
$$
Assume now that we are in the case~\eqref{item33}. Passing to the
quotient space $\mathbb{S}^1 / \varphi$ and arguing as above for
quotient dynamics $g_1,\dots,g_k$, the same behavior follows. That
is, there exists $\Omega\subset \Sigma_k^+$ with $\nu(\Omega)=1$
such that for every $\omega \in \Omega$ there is a point
$r=r(\omega)\in \mathbb{S}^1/\varphi$ with the property that the
length of any arc which does not contain  $r$ tends to 0 by
iteration of the orbital branch $g^n_\omega$. Since $\varphi$ is
an order $\ell>1$ orientation-preserving homeomorphism, the
equivalence class $r$ has exactly $\ell$ different representative
$r_i=r_i(\omega)\in \mathbb{S}^1$, for $i=1,\dots,\ell$. Moreover,
$\varphi^i(r_1)=r_i$, for $i=1,\dots,\ell$.

Let $U$ be a connected component of
$\mathbb{S}^1\setminus\{r_1,\dots,r_\ell\}$. Consider $x,y \in U$
and denote by $[x,y]$ the arc with endpoint $x$ and $y$ containing
in $U$. Observe that $r\not\in \pi([x,y])$ where
$\pi:\mathbb{S}^1\to \mathbb{S}^1/\varphi$ is the projection on
the quotient space. Hence, the length of this arc $\pi([x,y])$ in
the $\mathbb{S}^1/\varphi$ goes to zero by iteration $g^n_\omega$.
This means, in particular, that
$$
\lim_{n\to \infty} d(f^n_\omega(x),f^n_\omega(y))=0.
$$
This concludes the proof of the lemma.
\end{proof}

Now, we will prove Theorem~\ref{thmB}. Denoting by $|I|$ the
length
of an arc $I \subset \mathbb{S}^1$. 
\begin{proof}[Proof of Theorem~\ref{thmB}]
Suppose that we are in the case~\eqref{item11} of Antonov's
theorem. Then, there is a minimal homeomorphism among the
generator $f_1,\dots,f_k$ of the IFS. Otherwise, the $f_i$'s being
simultaneously conjugate to rational rotations. Hence, the group
generated by $f_1,\dots,f_k$ is finite and it cannot act minimally
on the circle. Proposition~\ref{pro:minimal-map} implies the
minimality of almost every orbital branch in this case.

Now, assume that there exists a homeomorphism  in
$\IFS(f_1,\dots,f_k)$ which is not conjugate to a rotation. By
Lemma~\ref{lem:Antonov}, there exists an integer $\ell\geq 1$ and
a subset $\Sigma \subset \Sigma_k^+$ with $\nu(\Sigma)=1$ such
that for every $\omega\in\Sigma$ there are points
$r_i=r_i(\omega)\in \mathbb{S}^1$, for $i=1,\dots,\ell$, such that
for any connected component $U$ of
$\mathbb{S}^1\setminus\{r_1,\dots,r_\ell\}$,  
\begin{equation}
\label{eq:as}
  \lim_{n\to\infty} d(f^n_\omega(x),f^n_\omega(y))=0, \quad \text{for every $x,y\in U$.}
\end{equation}
Let $\vartheta \in \Sigma_k^+$ has dense orbit under the Bernoulli
shift map and consider a point $x$ and an open set $I$ in
$\mathbb{S}^1$. We want to see that $
\mathcal{O}^+_\vartheta(x)\cap I\not=\emptyset$.
\begin{claim}
\label{clm2}There exists a finite word $\tau$  such that for each
$z\in \mathbb{S}^1$, there is $0\leq t=t(z) \leq |\tau|$ with the
property that $f_{\tau_t}\circ\dots\circ f_{\tau_1}(z)\in I$.
\end{claim}
By doing this, using the fact that the orbit of $\vartheta$ is
dense under the Bernoulli shift map $\sigma$, one can choose
$n\geq 1$ such that $[\sigma^{n}(\vartheta)]_i=\tau_i$, for
$i=1,\dots,|\tau|$. Taking $z=f_\vartheta^{n}(x)$ one has that
$$
f^{n+t(z)}_\vartheta(x)\in I.
$$
which proves the minimality of the orbital branch along
$\vartheta$. Since the set of all elements of $\Sigma_k^+$ with
dense orbit has full $\mathbb{P}$-probability, one can conclude
the minimality of almost every orbital branch. So, the proof of
Theorem~\ref{thmB} is done in the first part.
\begin{proof}[Proof of Claim~\ref{clm2}] Let $D$ be
a countable dense subset of $\mathbb{S}^1$. By~\cite{BV11}, there
is $\Omega(q) \subset \Sigma_k^+$ with $\nu(\Omega(q))=1$ such
that $\mathbb{S}^1=\overline{\mathcal{O}^+_\omega(q)}$, for every
$\omega \in \Omega(q)$. Consider
$$
\Omega=\bigcap_{q\in D} \Omega(q) \cap \Sigma.
$$
Observe that $\nu(\Omega)=1$ and $\mathcal{O}_\omega^+(q)$ is
dense in $\mathbb{S}^1$, for every $q\in D$ and every $\omega \in
\Omega$. Fix $\omega \in \Omega$. By~\eqref{eq:as}, for  every
$\varepsilon>0$, there exists $K=K(\omega,\varepsilon) \in
\mathbb{N}$ such that
\begin{equation}
\label{eq:e1}
  |f^n_\omega(U_i)|< \varepsilon, \quad \text{for every $n\geq K$ and
  $i=1,\dots,\ell$},
\end{equation}
where the arcs $U_i=U_i(\omega,\varepsilon)\subset \mathbb{S}^1$
are the connected components of the set
$$\mathbb{S}^1\setminus
(B_{\varepsilon}(r_1)\cup\dots\cup B_\varepsilon(r_\ell)). $$
Here,
$B_{r}(x)$ denotes the open ball of radius $r>0$ centered at the
point $x$.

In view of the backward minimality, one can recursively find maps
$h_1, \dots, h_\ell$ in $\IFS(f_1,\dots,f_k)$ such that
$h^{-1}_i\circ\dots \circ h^{-1}_1 (r_i)\in I$, for every
$i=1,\dots,\ell$. Now, consider $\varepsilon>0$ in such a way that
$$
   h_i^{-1}\dots \circ h^{-1}_1(B_\varepsilon(r_i)) \subset I, \quad \text{for every $i=1,\dots,\ell$.}
$$
Equation~\eqref{eq:e1} implies that by iteration of $f^n_\omega$
the arcs $U_i$ are contracted. Since each of these arcs has points
of $D$ and $\omega \in \Omega$, one can move these arcs around the
circle to any open set. In particular, there exist non-negative
integer numbers $n_1=n_1(\omega), \dots, n_\ell=n_\ell(\omega)$
with $n_1\geq K(\omega,\varepsilon)$ so that
$$
   f^{n_1+\dots+n_i}_\omega(U_i) \subset I, \quad \text{for every $i=1,\dots,\ell$.}
$$
Let $\gamma^i$ be the words corresponding to the maps $h_i$ for
$i=1,\dots,\ell$. Put
$$
\tau=\gamma^\ell\dots\gamma^1\omega_1\dots\omega_{n_1+\dots+n_\ell}.
$$
The word $\tau$ is the desired one in Claim~\ref{clm2}. To show
this it is sufficient to show that for any $z\in \mathbb{S}^1$
there is $0\leq t \leq |\tau|$ such that $
     f_{\tau_t} \circ\dots \circ f_{\tau_1}(z)\in I.
$
Consider the point
$$
y=h_1\circ \dots \circ
h_\ell(z)=f_{\tau_s}\circ\dots \circ f_{\tau_1}(z),
$$
where $s=|\gamma^1|+\dots+|\gamma^\ell|$. We have three cases:\\
\begin{itemize}
\item If $y\in U_1\cup \dots \cup U_\ell$,
then $f^{n_1+\dots+n_i}_\omega(y) \in I$,  for some $i\in
\{1,\dots,\ell\}$.
\item If $y\in B_\varepsilon(r_i)$ with $i\in\{1,\dots,\ell-1\}$, then
$$
    \hspace{1.2cm}h_{i+1}\circ\dots \circ h_\ell(z)=h^{-1}_{i}\circ \dots \circ h_1^{-1}(y)
    \in h^{-1}_{i}\circ \dots \circ h_1^{-1}(B_\varepsilon(r_i))\subset I.
$$
\item If $y\in B_\varepsilon(r_\ell)$, then
$z=h^{-1}_{\ell}\circ \dots \circ h_1^{-1}(y) \in
h^{-1}_{\ell}\circ \dots \circ h_1^{-1}(B_\varepsilon(r_\ell))
\subset I$.
\end{itemize}
In any case, one has that $f_{\sigma_t} \circ \dots \circ
f_{\sigma_1}(z)\in I$ for some $0\leq t \leq |\sigma|$ as we want.
\end{proof}

Continuing the proof of Theorem~\ref{thmB}, we focus on the second
part, which is  the density of periodic points. Let $J\subset
\mathbb{S}^1$ be any arbitrary small open interval. We will show
that there exists a periodic point of $\IFS(f_1,\dots,f_k)$ in
$J$.
\begin{claim}
\label{clm3}There is a fixed point $a$ of a map $g$ in
$\IFS(f_1,\dots,f_k)$ such that at least for one-side $a$ becomes
as an attracting point of $g$.
\end{claim}
\begin{proof}
We use the  notation applied in the proof of the first part.
By~\eqref{eq:e1} and the density of $D$ in $\mathbb{S}^1$, for
every $\omega\in\Omega$ and every $i\in\{1,\dots,\ell\}$, there is
a sufficiently large number $n$ such that $f^n_\omega(U_i)\subset
U_i$ and $|f^n_\omega(U_i)|<|U_i|$. By Brouwer's fixed-point
theorem, the mapping $g=f^n_\omega$ has a fixed point $a$ in $U_i$
satisfying the required attracting property.
\end{proof}
Forward minimality allows us to find $F\in \IFS(f_1,\dots,f_k)$
such that $F(J)\cap B$ has non-empty interior where $B$ is the
basin of attraction of $a$ for $g$. Let $V\subset J$ be a
non-degenerated closed arc such that $F(V) \subset B$. Again by
the forward minimality of the IFS, there is $G\in
\IFS(f_1,\dots,f_k)$ such that $G(a)$ belongs to the interior of
$V$. By the continuity of $G$, there is $\delta>0$ such that
$G((a-\delta,\,a+\delta))$ is contained in $V$. Now, since $F(V)$
is contained in the basin of $a$, there is $m\geq 0$ such that
$g^{m}\circ F(V) \subset (a-\delta,\,a+\delta)$ and so, $G\circ
g^{m}\circ F(V)\subset V$. By Brouwer's fixed-point theorem,
$G\circ g^m\circ F$ has a fixed point in $V\subset J$. This
implies the desired density of periodic points of
$\IFS(f_1,\ldots,f_k)$ and eventually ends the proof of the second
part of Theorem~\ref{thmB}.
\end{proof}
\begin{remark}
\label{rem} The proof of the second part of Theorem~\ref{thmB} can
be improved provided the generators $f_1,\dots,f_k$ are
$C^1$-diffeomorphisms and there is a hyperbolic attracting fixed
point of some map in $\IFS(f_1,\dots,f_k)$. This improvement
involves taking a large iteration $g^m$ in such a way that $G\circ
g^m\circ F$ is contracting on $V$. After that, one uses Banach's
fixed-point theorem concluding the existence of a unique
hyperbolic attracting fixed point of $G\circ g^m\circ F$ in $V$.
In fact, we obtain something more which is the density of
hyperbolic attracting periodic points of the IFS. Similar argument
also runs to the backward IFS getting the same result for
repellers.
\end{remark}
\section{Robust minimal IFSs: Proof of Theorem~\ref{thmC}}
\label{s:4}
%
Let $g_1, g_2 \in \mathrm{Diff}^1(\mathbb{S}^1)$ be, respectively,
the irrational rotation $x\to x+\alpha$ and an
orientation-preserving diffeomorphism which is not conjugate to a
rotation. This implies that $\IFS(g_1,g_2)$ is not conjugate to a
semigroup of rotations. As we have shown in the proof of
Theorem~\ref{thmB}, in this case, there exists $g \in
\IFS(g_1,g_2)$ with a fixed point $a$ such that at least for
one-side $a$ becomes an attracting point of $g$. Without lose of
generality, we suppose that $g_2=g$ and $A=(a,a+\varepsilon)$ is
its local basin of attraction, i.e., $0<Dg_2(x)<1$, for every
$x\in A$. Consider $\delta>0$ such that
$|\delta-\varepsilon|>|a-g_2(a+\varepsilon)|$ and set
\begin{equation}
\label{intBD}
B=\big(g^2_2(a+\varepsilon),\,g_2(a+\varepsilon)\big) \quad
\text{and} \quad D=\big(a,\,g_2(a+\varepsilon)\big).
\end{equation}
Observe that since $g_1$ is an irrational rotation, there exist
$n_1,\dots,n_k\in \mathbb{N}$ such that for $h_i=g_1^{n_i}\circ
g_2$ it holds
\begin{enumerate}
\item \label{eq:blender1}
$\overline{B} \subset h_1(B) \cup \dots \cup h_k(B)$,
\item \label{eq:blender2}
$g_1^{n_i}(\overline{D}) \subset (a+\delta,\,a+\varepsilon)$, for
every $i=1,\dots,k$,
\item \label{eq:blender3}
there is $\lambda<1$ such that $Dh_i(x)<\lambda$, for every $x\in
(a+\delta,a+\varepsilon)$ and $i=1,\dots,k$.
\end{enumerate}
The third term holds by the equality $Dg_1^{n_i}=1$.

Using the mapping $g_1$, one can cover the whole circle by
finitely many iterations of $B$. That is, there exist times $m_1,
\ldots,m_s$ and $\tilde{m}_1, \ldots,\tilde{m}_r$ such that for
$T_i=g_1^{m_i}$ and $S_i =g_1^{\tilde{m}_i}$ it holds
\begin{equation}
\label{eq:cover}
   \mathbb{S}^1=\bigcup_{i=1}^s T_i(B)=\bigcup_{i=1}^r S^{-1}_i(B).
\end{equation}
Now, we shall prove that conditions ~\eqref{eq:blender1},
\eqref{eq:blender2}, \eqref{eq:blender3} and~\eqref{eq:cover}
imply that the IFS generated by $g_1$ and $g_2$ is $C^1$-robustly
minimal. Namely, we want to prove that for every $\Phi=(f_1,f_2)$
close enough to $(g_1,g_2)$ in the $C^1$-topology,
$$
  \mathbb{S}^1=\overline{\mathrm{Orb}^+_\Phi(x)},
  \quad \text{for every $x\in \mathbb{S}^1$}
$$
where, as before, $\mathrm{Orb}_\Phi^{+}(x)=\{h(x):
h\in\IFS(f_1,\dots,f_k)\}$.

First of all, observe that~\eqref{eq:blender1},
\eqref{eq:blender2}, \eqref{eq:blender3} and~ \eqref{eq:cover} are
$C^1$-robust conditions i.e., there are $C^1$-neighborhoods
$\mathcal{V}_i$ of $g_i$, for $i=1,2$, such that for each
$(f_1,f_2) \in \mathcal{V}_1\times \mathcal{V}_2$ there are maps
$\tilde{h}_1, \dots, \tilde{h}_k$, $\tilde{T}_1, \dots,
\tilde{T}_m$, $\tilde{S}_1,\dots,\tilde{S}_r$ in $\IFS(f_1,f_2)$
so that conditions~ \eqref{eq:blender1}, \eqref{eq:blender2},
\eqref{eq:blender3} and~\eqref{eq:cover} hold for these map for
the same $B$ and $D$. Hence, to prove the robust minimality of the
IFS generated by $g_1$ and $g_2$, it suffices to show that
\eqref{eq:blender1}, \eqref{eq:blender2}, \eqref{eq:blender3}
together with  \eqref{eq:cover} imply the forward minimality.

We begin by a simple observation that by~\eqref{intBD} and
\eqref{eq:blender2}, for every $i_1\dots i_n$, with $i_j \in
\{1,\dots,k\}$, it holds
\begin{equation}
\label{first-hi}
  h_{i_1}\circ \dots \circ h_{i_n}(B) \subset (a+\delta,a+\varepsilon).
\end{equation}
The rest of the proof will be done through a few steps for which
we need a bit of notation. For any $n>1$, define recursively sets
$$
    B^n_{i_1\ldots i_n} = h_{i_n}(B^{n-1}_{i_1\ldots i_{n-1}})=h_{i_n} \circ \dots \circ h_{i_1}
    (B),
$$
for $i_j=1,\ldots,k$ and $j=1,\ldots,n$. The proof of the lemma
below is short and elementary yet releases us from some notation
inconsistency in this context.
\begin{lemma}
\label{lemma2} For every $n\in \mathbb{N}$ it holds
\begin{align*}
B^n_{i_2\ldots i_{n+1}} \subset \bigcup_{i_1=1}^k
B^{n+1}_{i_1i_2\ldots i_{n+1}}  \qquad \text{and} \qquad B \subset
\bigcup_{i_1,\ldots,i_{n+1}=1}^k B^{n+1}_{i_1\ldots i_{n+1}}.
\end{align*}
\end{lemma}
\begin{proof}
The proof is by induction on $n$. First, we  show that
$$
B_{i_2}^1 \subset \bigcup_{i_1=1}^k B^2_{i_1 i_2} \quad \text{and}
\quad B\subset \bigcup_{i_1,i_2=1}^k B_{i_1i_2}^2.
$$
By condition~\eqref{eq:blender1}, $B \subset B_{1}^1\cup \dots
\cup B_{k}^1$, and so,
\begin{align*} \label{eq1}
   \bigcup_{i_1=1}^k B^2_{i_1 i_2} &= \bigcup_{i_1=1}^k h_{i_2} ( B^1_{i_1})
   = h_{i_2} (\bigcup_{i_1=1}^{k}B^1_{i_1})  \supset h_{i_2}(B) = B^1_{i_2}.
\end{align*}
From this one gets that
$$
  \bigcup_{i_1,i_2=1}^k B_{i_1i_2}^2 \supset \bigcup_{i_2=1}^k B_{i_2}^1 \supset B.
$$
Now, we assume that the lemma holds for $n-1$ and we shall prove
it for $n$. In the same way as before,
\begin{align*}
   \bigcup_{i_1=1}^k B^{n+1}_{i_1\dots i_{n+1}} &=
   \bigcup_{i_1=1}^k h_{i_{n+1}}(B^{n}_{i_1\dots i_{n}})
      = h_{{i_{n+1}}}(\bigcup_{i_1=1}^{k}B_{i_1\dots i_n}^n).
\end{align*}
By hypothesis of the induction, one has that  $B^{n-1}_{i_2\ldots
i_{n}} \subset \cup_{i_1=1}^k B^{n}_{i_1i_2\dots i_{n}}$ and so
$$
\bigcup_{i_1=1}^k B^{n+1}_{i_1\dots i_{n+1}} \supset h_{{i_{n+1}}}
(B^{n-1}_{i_2\dots i_{n}})=B^{n}_{i_2\dots i_{n+1}}. $$

For the second inclusion of the lemma, we first note that for
every $1\leq \ell \leq n$ and every $i_j=1,\dots,k$ with
$j=2,\dots,\ell+1$,
$$
B^{\ell}_{i_2\dots i_{\ell+1}} \subset \bigcup_{i_1=1}^k
B^{\ell+1}_{i_1i_2\dots i_{\ell+1}}.
$$
Hence,
$$
  \bigcup_{i_1,\dots,i_{n+1}=1}^k B^{n+1}_{i_1\dots i_{n+1}} \supset
  \bigcup_{i_2,\dots,i_{n+1}=1}^k B^{n}_{i_2\dots i_{n+1}}
  \supset \dots \supset \bigcup_{i_{n+1}=1}^k B^1_{i_{n+1}} \supset
  B,
$$
and the proof of the lemma is completed.
\end{proof}
\begin{lemma}
\label{lemma3} Under the
conditions~\eqref{eq:blender1},~\eqref{eq:blender2}
and~\eqref{eq:blender3}, for every $x\in \overline{B}$ there is a
sequence $(i_j)_{j>0}$, $i_j \in \{1,\ldots,k\}$ such that
$$
  x= \lim_{n\to\infty} h_{i_1}\circ
  \dots \circ h_{i_n}(y), \quad \text{for every} \ y\in B.
$$
\end{lemma}
\begin{proof}
Since $\overline{B} \subset  B^1_{1}\cup \dots \cup B^1_{k}$, for
every $x\in \overline{B}$ there is $i_1 \in \{1,\ldots,k\}$ such
that $x\in B^1_{i_1}$. We now proceed recursively. For $n>1$,
suppose that $i_j \in \{1,\ldots,k\}$, for $j=1,\ldots,n$, has
chosen in such a way that $ x \in B^n_{i_n\ldots i_1}$. Using the
inclusion $B^{n}_{i_n\ldots i_1} \subset \cup_{j=1}^k
B^{n+1}_{ji_n\ldots i_1}$, given by Lemma~\ref{lemma2}, one can
choose $i_{n+1} \in \{1,\ldots,k\}$ such that $x\in
B^{n+1}_{i_{n+1}i_n\ldots i_1}$. From this, a sequence
$i=i_1i_2\ldots = (i_j)_{j>0}$ of positive integers can be
constructed so that $ x \in B^n_{i_n\ldots i_1}$, for every $n\geq
1$. Hence, one gets
$$
x\in \bigcap_{n\geq 1} B^n_{i_n \ldots i_1} = \bigcap_{n\geq 1}
h_{i_1}\circ \dots \circ h_{i_n}(B) =\bigcap_{n\geq 1} A_n,
$$
where  $A_n=\cap_{\ell=1}^n h_{i_1}\circ \dots \circ
h_{i_\ell}(B)$, for every  $n\in \mathbb{N}$.

Since $ A_{n+1} \subset A_{n} \subset h_{i_1}\circ \dots \circ
h_{i_n}(B)$, using the mean-value theorem, there are $z_j \in B$,
for $j=1,\dots,n$, such that
\begin{align*}
 \mathrm{diam}&(A_n) \leq \mathrm{diam}(h_{i_1}\circ
 \dots \circ h_{i_n}(B))  \\
 &\leq \prod_{j=1}^{n-1} Dh_{i_j}(h_{i_{j+1}}\circ
 \dots\circ h_{i_n}(z_j)) \cdot Dh_{i_n}(z_n) \cdot \mathrm{diam}(B).
\end{align*}
According to~\eqref{first-hi}, $h_{i_{j+1}}\circ \dots\circ
h_{i_n}(z_j) \in (a+\delta,a+\varepsilon)$ and thus, condition
\eqref{eq:blender3} implies that $\mathrm{diam}(A_n)\leq \lambda^n
\mathrm{diam}(B)$. Consequently, $A_n$ being a nested sequence of
sets whose diameters goes to zero and hence one gets
$$
 \{ x \}= \bigcap_{n\geq 1} B^n_{i_n\ldots i_1}=
 \bigcap_{n\geq 1} h_{i_1}\circ \dots \circ h_{i_n}(B).
$$
This implies that for every $y\in B$ and every natural number $n$,
$$
|h_{i_1} \circ \dots \circ h_{i_n}(y)- x| \leq
\mathrm{diam}(h_{i_1}\circ \dots \circ h_{i_n}(B)) \leq  \lambda^n
\mathrm{diam}(B).
$$
This completes the proof of the lemma.
\end{proof}
\begin{proof}[Proof of Theorem~\ref{thmC}]
Now, we show the minimality of any IFS generated by two
diffeomorphisms $g_1$, $g_2$ satisfying
conditions~\eqref{eq:blender1}, \eqref{eq:blender2},
\eqref{eq:blender3} and~\eqref{eq:cover}. Fix a point $x\in
\mathbb{S}^1$ and an open set $I\subset \mathbb{S}^1$. The
previous lemma implies that
$$
B \subset \overline{\mathrm{Orb}_\Phi^+(x)}, \quad \text{for every
$x\in B$.}
$$
By~\eqref{eq:cover}, we find  $i\in\{1,\dots,r\}$ such that
$S_i(x)\in B$. Similarly, $B\cap T_j^{-1}(I)$ contains an open set
for some $j\in\{1,\dots,s\}$. Using the density of the orbits in
$B$, one can find $h\in\IFS(g_1,g_2)$ such that $h\circ S_i (x)\in
T_j^{-1}(I)$. Thus, $T_j\circ h \circ S_i(x)\in I$. This shows
that $\IFS(g_1,g_2)$ is minimal as we want.

In order to show the robust backward minimality, we observe that
we can obtain~\eqref{eq:blender1},
~\eqref{eq:blender2},~\eqref{eq:blender3} and~\eqref{eq:cover} but
now for  $\IFS(g_1^{-1},g_2^{-1})$. Thus, applying the same
argument for $\IFS(g_1^{-1},g_2^{-1})$, we conclude the robust
minimality of this system. Finally, to conclude
Theorem~\ref{thmC}, it suffices to apply Theorem~\ref{thmB} and
Remark~\ref{rem}.
\end{proof}

\section*{Acknowledgments}
We are grateful to Lorenzo J. D\'iaz  for his assistance and
review of preliminary versions of this paper and Ale Jan Homburg
and Malicet Dominique for their comments. During the preparation
of this article the authors were partially supported by the
following fellowships: P. G. Barrientos by MTM2011-22956 project
(Spain) and CNPq post-doctoral fellowship (Brazil); A. Fakhari by
grant from IPM. (No. 91370038). A. Fakhari and A. Sarizadeh thank
ICTP for their hospitality during their visit.


\medskip
Received August  2013; revised January 2014.
\medskip
\end{document}